\documentclass[11pt,twoside]{article}

\usepackage{mathrsfs,amsfonts,amsmath,amssymb}
\usepackage{latexsym}
\usepackage{comment}

\newtheorem{satz}{Theorem}

\newtheorem{theorem}[satz]{Theorem}
\newtheorem{lemma}[satz]{Lemma}

\newtheorem{remark}[satz]{Remark}
\newtheorem{example}[satz]{Example}

\def\L{\Lambda}

\def\Z{\mathbb {Z}}
\def\R{\mathbb {R}}
\def\F{\mathbb {F}}
\def\E{\mathsf{E}}

\def\C{\mathbb{C}}

\def\d{\delta}
\def\o{\omega}
\def\({\big (}
\def\){\big )}

\def\le{\leqslant}
\def\ge{\geqslant}
\def\_phi{\varphi}
\def\eps{\varepsilon}

\def\Gr{{\mathbf G}}
\def\FF{\widehat}

\def\ov{\overline}

\def\D{\Delta}

\def\supp{\mathsf{supp}}

\def\T{\mathsf{T}}

\def\C{\mathbb{C}}

\newcommand{\bp}{\bigskip}

\author{I.D. Shkredov}
\title{
Uncertainty for convolutions of sets
}
\date{}
\begin{document}
	\maketitle


\begin{center}
	Annotation.
\end{center}

{\it \small
    The paper obtains the optimal form of the uncertainty principle in the special case of convolution of sets.
}
\\

\section{Introduction}

Let $\Gr$ be an
abelian group and  $\FF{\Gr}$  its dual group.
For any function $f:\Gr \to \mathbb{C}$ and $\xi \in \FF{\Gr}$ define the Fourier transform of $f$ at $\xi$ by the formula 
\begin{equation}\label{f:Fourier_representations}
\FF{f} (\xi) = \sum_{g\in \Gr} f(g) \ov{\xi (g)} \,.
\end{equation}
A 
number 
of uncertainty principles on $\Gr$ assert, roughly, that a function on $\Gr$ and its Fourier transform cannot be simultaneously highly concentrated. 
This topic is quite  popular 
see, e.g., papers \cite{alagic2008uncertainty}, \cite{biro2021uncertainty}, \cite{meshulam2006uncertainty},
\cite{Tao_uncertainty}. 
The most basic example is an inequality which connects the sizes of the supports of $f$ and $\FF{f}$ to the size of the finite group $\Gr$, namely, 
\begin{equation}\label{f:uncertainty_intr}
    |\supp f| \cdot |\supp \FF{f}| \ge |\Gr| \,,
\end{equation}
and, of course, it is worth mentioning the classical Heisenberg inequality for $\Gr = \R$, which  states  that for any $a,b\in \R$ one has 
\begin{equation}\label{f:uncertainty_Heisenberg_intr}
    \int_{\R} (x-a)^2 |f(x)|^2\, dx \cdot \int_{\R} (\xi-b)^2 |\FF{f} (\xi)|^2\, d\xi \ge \frac{\|f\|^4_2}{16\pi^2} \,.
\end{equation}
For more background on the Fourier transform on abelian groups and the uncertainty principle we refer to \cite{terras1999fourier} and the excellent surveys 
\cite{folland1997uncertainty} and \cite{wigderson2021uncertainty}.

This paper considers the case when the function $f$ has a special form, namely, $f$ is the convolution of the characteristic function of a set $A \subseteq \Gr$ (all required definitions can be found in Section  \ref{sec:def}). 
This allows us to obtain a new optimal form of the uncertainty principle in terms of the difference set $A-A$ (we also consider the asymmetric case, see Theorem \ref{t:rho,M} below).
Given two sets $A,B\subseteq \Gr$, define  
the {\it sumset} 
of $A$ and $B$ as 
$$A+B:=\{a+b ~:~ a\in{A},\,b\in{B}\}\,.$$
In a similar way we define the {\it difference sets} and the {\it higher sumsets}, e.g., $2A-A$ is $A+A-A$.
Consider two quantities 
\begin{equation}\label{def:rho}
    \rho (A) := \max_{x\neq 0} |A \cap (A+x)| \,,
\end{equation}
and 
\begin{equation}\label{def:M(A)}
    M (A) := \max_{\xi\neq 0} |\FF{A} (\xi)| \,,
\end{equation}
where we use the same capital letter to denote a set $A\subseteq \Gr$ and   its characteristic function $A: \Gr \to \{0,1 \}$.

\begin{theorem}
     Let $\Gr$ be a finite abelian group, $A\subseteq \Gr$, $|A|=\d |\Gr|$,  and let $|A-A| = K|A|$. 
     Suppose that  $|A| \to \infty$, $\log^2 K = o(\log |A|)$ and 
\begin{equation}\label{cond:Kd}
    K^2 \delta = o(1)\,.
\end{equation}
    Then  
\begin{equation}\label{f:rho,M_cor}
     M^2 (A) \rho (A) \ge \frac{|A|^3}{K} \cdot (1-o(1)) \,.
\end{equation}
\label{t:rho,M_cor}
\end{theorem}
In other words, if $f$ is the convolution of the characteristic function of a set $A \subseteq \Gr$, then 
\begin{equation}\label{f:rho,M_cor_via_f}
    \max_{x\neq 0} |f(x)| \cdot \max_{\xi\neq 0} |\FF{f}(\xi)| \ge  \frac{|A|^{3}}{K} \cdot (1-o(1)) \,.
\end{equation}
It is easy to see that bound \eqref{f:rho,M_cor} is tight (also, see Remark \ref{r:H+L_energy} below). 
Indeed,  take $\Gr = \F_2^n$, let $H$ be a proper subspace of $\Gr$ and let $A\subseteq H$ be a random set of positive density. Then with high probability 
$$K|A|:= |A-A| = |H| (1-o(1)),\, \quad \quad  \rho (A) \le \frac{|A|}{K} \cdot (1-o(1))$$ and $\E(A,H)= |A|^2 |H|$. 
The last 
identity 
implies $M(A) = |A| (1-o(1))$ 
and the same 
asymptotic formula 
follows from  our bound \eqref{f:rho,M_cor}. 
As for condition \eqref{cond:Kd}, we definitely need that  $K\d = o(1)$ (consider the case of a random set $A$).
Finally, note that in our proof we make extensive use of the form of the function $f$ (for example, we widely use of the concept of higher sumsets/higher energies, see \cite{SS_higher}) 
and it seems like  any reasonable analogue of \eqref{f:rho,M_cor_via_f} for general $f$ 
fails, 
see  
Remark \ref{r:non_general} below.

\bp

We thank Vsevolod F. Lev for useful  discussions and valuable comments.


\section{Definitions and notation}
\label{sec:def}

Let $\Gr$ be a finite abelian group and we denote the cardinality of $\Gr$ by $N$. 
Given a set $A\subseteq \Gr$ and a positive integer $k$, we put 
$$
    \Delta_k (A) := \{ (a,a, \dots, a) ~:~ a\in A \} \subseteq \Gr^k \,.
$$
Also, let 
$\Delta_k (x) := \Delta_k (\{ x \})$, $x\in \Gr$.
Now we have 
\begin{equation}\label{def:A-A_intersection}
    A-A := \{ a-b ~:~ a,b\in A\} = \{ s\in \Gr ~:~ A\cap (A+s) \neq \emptyset \} \,.
\end{equation}
A natural generalization of the last formula in \eqref{def:A-A_intersection} 
is the set 
\begin{equation}\label{def:A-A_m_intersection}
    \{ (x_1, \dots, x_k) \in \Gr^k ~:~ A\cap (A+x_1) \cap \dots \cap (A+x_k) \neq \emptyset \} 
    =
    A^k - \Delta_k (A) \,,
\end{equation}
which is called the {\it higher difference set} (see \cite{SS_higher}). 
Given a positive integer $k$,  we put as in \cite{LS_popularity}   
\[
    R^{(k+1)}_A (x_1,\dots, x_k) = |A\cap (A+x_1) \cap \dots \cap (A+x_k)| \,, \quad \quad  x_1,\dots, x_k \in \Gr \,.
\]
Thus, $(x_1, \dots, x_k) \in A^k - \Delta_k (A)$ iff $R^{(k+1)}_A (x_1,\dots, x_k) >0$. 
For any two sets $A,B \subseteq \Gr$ the {\it additive energy} of $A$ and $B$ is defined by
$$
\E (A,B) = \E^{} (A,B) = |\{ (a_1,a_2,b_1,b_2) \in A\times A \times B \times B ~:~ a_1 - b^{}_1 = a_2 - b^{}_2 \}| \,.
$$
If $A=B$, then  we simply write $\E^{} (A)$ for $\E^{} (A,A)$.
Similarly, 
\begin{equation}\label{def:T_k_intr}
    \T_k^{} (A) := |\{ (a_1, \dots, a_k,a'_1, \dots, a'_k) \in A^{2k} ~:~
    a_1 + \dots + a_k = a'_1 + \dots + a'_k \} |\,.
\end{equation}
Also, one can define  the {\it higher energy} (see \cite{SS_higher})  
\begin{equation}\label{f:E_k2}
      \E_{k}(A) =  \sum_{x} (A\circ A)^k (x) = \sum_{x_1,\dots,x_{k-1}} R^{(k)} (x_1,\dots,x_{k-1})^2 := \E_{k,2} (A) \,.
\end{equation}
It is known and easy to check (see, e.g., \cite[proof of Theorem 4]{LS_popularity}) that $\E_{k,l} (A) = \E_{l,k} (A)$ for any integers $k,l>1$.
Note that the first formula in \eqref{f:E_k2} can be thought of as the definition of $\E_k (A)$ for any real $k>1$.

Let $\FF{\Gr}$ be its dual group.
For any function $f:\Gr \to \mathbb{C}$ and $\xi \in \FF{\Gr}$ we define its Fourier transform using the formula \eqref{f:Fourier_representations}.
The Parseval identity is 
\begin{equation}\label{F_Par}
    N\sum_{g\in \Gr} |f(g)|^2
        =
            \sum_{\xi \in \FF{\Gr}} \big|\widehat{f} (\xi)\big|^2 \,.
\end{equation}
If $f,g : \Gr \to \C$ are some functions, then 
$$
    (f*g) (x) := \sum_{y\in \Gr} f(y) g(x-y) \quad \mbox{ and } \quad (f\circ g) (x) := \sum_{y\in \Gr} f(y) g(y+x) \,.
$$
One has 
\begin{equation}\label{f:F_svertka}
    \FF{f*g} = \FF{f} \FF{g} \,.
\end{equation}
Having a function $f:\Gr \to \C$ and a positive integer $k>1$, we  write $f^{(k)} (x) = (f\circ f \circ \dots \circ f) (x)$, where the convolution $\circ$ is taken $k-1$ times. 
For example, $A^{(4)} (0) = \E(A)$. 
Also, let $\T_k (f) = \sum_{x} (f^{(k)} (x))^2$.

\bp

We need a particular case of Corollary 1 from  \cite{sh_new_ineq}.

\begin{lemma}
    Let $k\ge 2$ be an even integer,  $l\ge 2$ be any integer and $A\subseteq \Gr$ be a set. 
    Then 
\begin{equation}\label{f:R_T_k} 
    \sum_{x} (A^{(k)} (x))^l = \T_{k/2} (R^{(l)}_A) \,.
\end{equation}
\label{l:R_T_k}
\end{lemma}

Also, we need 
the generalized triangle inequality \cite[Theorem 7]{SS_higher}.

\begin{lemma}
    Let $k_1,k_2$ be positive integers, $W\subseteq \Gr^{k_1}$, $Y\subseteq \Gr^{k_2}$ and $X,Z \subseteq \Gr$. 
    Then
\begin{equation}\label{f:gen_triangle_S}
    |W\times X| |Y-\Delta_{k_2} (Z)| \le |W\times Y \times Z - \Delta_{k_1+k_2+1} (X)| \,. 
\end{equation}
\label{l:gen_triangle_S}
\end{lemma}

We have defined the quantity $\rho (A)$ and $M(A)$ in formulae \eqref{def:rho}, \eqref{def:M(A)} above. 
More generally,
consider 
\[
    \rho^{(k)} (A) := \max_{x\neq 0} A^{(k)} (x) \,.
\]
Thus $A^{(2)} (x) = (A\circ A) (x) = |A\cap (A+x)|$. 
Similarly to formula \eqref{def:rho} we put for any positive integer $l>1$
\[
    \rho_l (A) := \max_{|X|=l} |A_X| \,,
\]
where $A_X:= (A+x_1) \cap (A+x_2) \cap \dots \cap (A+x_l)$ and $X= \{x_1,\dots, x_l\}$ is a set of the cardinality $l$. Thus $\rho_2 (A) = \rho (A)$. 
If $X=\{0,x\}$, where $x\in \Gr$, then we write $A_{\{0,x\}} = A_x = A\cap (A+x)$. 
The inclusion of Katz--Koester \cite{Katz-Koester} is 
\begin{equation}\label{f:Katz-Koester}
    B+A_x \subseteq (A+B)_x \,.
\end{equation}

The signs $\ll$ and $\gg$ are the usual Vinogradov symbols. 
All logarithms are to base $e$.

\section{The proof of the main result}
\label{sec:proof}

We are ready to obtain our first main result, which implies Theorem \ref{t:rho,M_cor} from the introduction.


\begin{theorem}
    Let $\Gr$ be a finite abelian group, $A, B\subseteq \Gr$, $|A|=\d N$, $\d \in (0,1]$, further $|B|=\zeta |A|$,  $|A-A| = K|A|$, $|B-B| = K_*|B|$, and let $|A+B|=\o |A-A|$, where $\zeta, \o \in \R^{+}$. 
    Suppose that $|A| \ge (2K \cdot \max\{1,\o\})^8$.
    Then 
\begin{equation}\label{f:rho,M}
   M^2 (B) \rho (A) \ge \frac{|A|^2 |B|^2}{|A+B|} 
    \cdot \left( 1-  
     \frac{6 \log K \cdot \log (\zeta KK^*)}{\log |A|} 
    - (\o K)^2 \d \right) 
   \,,
\end{equation}
    and, similarly, 
\begin{equation}\label{f:rho,M_2}
    M^{} (A-A) \rho^2 (A) \ge \frac{|A|^3}{K} 
    \cdot \left( 1-  \frac{16 \log^2 (2K)}{\log |A|}  - K^3 \d \right) \,.
\end{equation}  
\label{t:rho,M}
\end{theorem}
\begin{proof}
    Let $\rho = \rho (A)$, $M=M(B)$, $D=A-A$, $S=A+B$, $|D| = K|A|$, $|S| = \o |D|$, and $\kappa := K \d \o \le 1$.
    By the usual triangle inequality for sets (take $k_1=k_2=1$ in Lemma \ref{l:gen_triangle_S}), we have $|A| |B-B| \le |A+B|^2$ and hence 
\begin{equation}\label{tmp:zeta_upper}
    \zeta K_* \le \o^2 K^2 \,.
\end{equation}
    Also, let $n\ge 2$ be a positive integer and put 
\begin{equation}\label{def:eps_n}
    |A^{n-1} - \D_{n-1} (A)| = \eps_{n-1} |D|^{n-1} \,,
\end{equation}
    where 
\begin{equation}\label{f:eps_n_lower}
    K^{-n} \le \eps_n \le 1 \,.
\end{equation}    
    Then by the Cauchy--Schwarz inequality and 
    the second equality in 
    \eqref{f:E_k2}  one has 
\begin{equation}\label{tmp:05.04_2}
    |A|^{2n} \le  |A^{n-1} - \D_{n-1} (A)| \E_{n} (A)
    \le \eps_{n-1} |D|^{n-1} (\rho^{n-1} |A|^2 + |A|^n) 
\end{equation}
\begin{equation}\label{tmp:05.04_3}
    =
    \eps_{n-1} |D|^{n-1} \rho^{n-1} |A|^2  +  \eps_{n-1} K^{n-1} |A|^{2n-1} \,.
\end{equation}
    Thus, we get 
\begin{equation}\label{f:rho}
    \rho \ge \eps^{-\frac{1}{n-1}}_{n-1} \cdot \frac{|A|}{K} \left( 1-  \frac{\eps_{n-1} K^{n-1}}{|A|} \right)^{1/(n-1)} 
    >
    \eps^{-\frac{1}{n-1}}_{n-1} \cdot \frac{|A|}{K} \left( 1-  \frac{2\eps_{n-1} K^{n-1}}{(n-1)|A|} \right)  \,,
\end{equation}
    provided $n \le 2^{-1} \log_K |A|$. 
    Now using Lemma \ref{l:gen_triangle_S}, we obtain for an arbitrary  set  $M\subseteq \Gr$ and any positive integer $m$
\[
    \eps_{m-1} |M|  |D|^{m-1} = |M| |A^{m-1} - \D_{m-1} (A)|
    \le
    |A^m - \D_m (M)| 
    \le 
    |A-M|^m \,,
\]
    and thus replacing $M$ to $(-M)$, we derive 
\begin{equation}\label{f:growth}
    |A+M| \ge \eps^{1/m}_{m-1} \cdot \left( \frac{|M|}{|D|} \right)^{1/m} |D| \,.
\end{equation}
    Finally,  we repeat the argument from \cite{LS_2.6}. 
    Using \eqref{f:growth}, the inclusion \eqref{f:Katz-Koester} and the H\"older inequality, one 
    obtains 
\[
    \sigma:= \sum_x |B_x| |S_x| \ge \sum_x |B_x| |A+B_x|
    \ge 
    \eps^{1/m}_{m-1} |D|^{1-1/m} \sum_x |B_x|^{1+1/m}
\]
\begin{equation}\label{tmp:sigma_08.04}
    \ge 
    \eps^{1/m}_{m-1} K |A| |B|^2 (\zeta K K_*)^{-1/m} \,.
\end{equation}
    Applying the Fourier transform, formula \eqref{f:F_svertka} and the Parseval identity \eqref{F_Par}, we get 
\[
    \sigma = N^{-1} \sum_\xi |\FF{B} (\xi)|^2 |\FF{S} (\xi)|^2
    \le \frac{\o^2 K^2 |B|^2 |A|^2}{N} + M^2 \cdot (\o K|A| - \o^2 K^2 |A|^2/N) 
\]
\begin{equation}\label{tmp:07.04_1}
    <  \o \kappa K |A|^{} |B|^2  +  \o K M^2|A| \,. 
\end{equation}
    Combining the last two bounds, we 
    have 
\begin{equation}\label{f:M}
    \o M^2 > \left( \eps^{1/m}_{m-1} (\zeta K K_*)^{-1/m} - \o \kappa \right) |B|^2  \,.
\end{equation}
    Taking $m=n = [2^{-1} \log_K |A|] \ge 2$ and multiplying \eqref{f:rho}, \eqref{f:M}, we obtain in view of estimates \eqref{tmp:zeta_upper}, \eqref{f:eps_n_lower} that 
\[
    \o M^2 \rho \ge \frac{|A| |B|^2}{K} \cdot \left( 1-  \frac{2\eps_{n-1} K^{n-1}}{(n-1)|A|} \right)
    \left( \eps^{-\frac{1}{n(n-1)}}_{n-1}  (\zeta K K_*)^{-1/n} -  \eps^{-\frac{1}{n-1}}_{n-1} \o \kappa \right) 
\]
\[
    > 
    \frac{|A| |B|^2}{K} \cdot  \left( 1-  \frac{2}{\sqrt{|A|}} \right)  \left( 1-  \frac{\log (\zeta KK^*)}{n}  - K \kappa \o \right) 
\]
\[
    \ge
    \frac{|A| |B|^2}{K} \cdot  \left( 1-  \frac{2}{\sqrt{|A|}} \right)  \left( 1-  \frac{4 \log K \cdot \log (\zeta KK^*)}{\log |A|}  - K \kappa \o \right) 
\]
\begin{equation}\label{tmp:05.04_1}
    > \frac{|A| |B|^2}{K} \cdot \left( 1-  \frac{6 \log K \cdot \log (\zeta KK^*)}{\log |A|}  - K \kappa \o \right) 
\end{equation}
    as required.

    It remains to obtain \eqref{f:rho,M_2}. 
    In this case we use calculations similar to \eqref{tmp:sigma_08.04} (now $B=A$) and consider
\[
    \sigma_* = \sum_{x\in D} |D_x| \ge \sum_{x\in D} |A-A_x| \ge \eps_{m-1}^{1/m} |D|^{1-1/m} \sum_{x\in D} |A_x|^{1/m} 
\]
\[
    > 
    \eps_{m-1}^{1/m} K \rho^{-1} |A|^3 (|D|/\rho)^{-1/m} 
    - \eps_{m-1}^{1/m} K \rho^{-1} |A|^2 (|D|/\rho)^{-1/m} \,.
\]
    As in  \eqref{tmp:07.04_1} putting $\kappa_* = \d K^2\rho |A|^{-1} \le \d K^2$, we get 
\[
    \sigma_* = N^{-1} \sum_\xi \FF{D} (\xi) |\FF{D} (\xi)|^2
    \le \frac{K^3 |A|^3}{N} + M(D) \cdot (K|A| - K^2 |A|^2/N) 
\]
\[
    <  \kappa_* K |A|^3 \rho^{-1}  +  K M (D) |A| \,. 
\]
    Combining the last two bounds, we derive 
\[
    M(D) \rho > \left( \eps^{1/m}_{m-1} (|D|/\rho)^{-1/m} (1-|A|^{-1}) - \kappa_* \right) |A|^2 
    \ge \left( \eps^{1/m}_{m-1} (2K)^{-2/m} (1-|A|^{-1}) - \kappa_* \right) |A|^2  \,.
\]
    Here we have used 
    the simple bound 
    $\rho \ge |A|/(2K)$. 
     Taking $m=n = [2^{-1} \log_K |A|] \ge 2$   and multiplying the last estimate by \eqref{f:rho}, we obtain the desired bound \eqref{f:rho,M_2}. 
    This completes the proof. 
$\hfill\Box$
\end{proof}

\bp 

Let us discuss obtained inequalities  \eqref{f:rho,M_cor}, \eqref{f:rho,M}. 
First of all, 
one can check 
that the both bounds 
are tight, see the construction after 
formula 
\eqref{f:rho,M_cor_via_f}. 
Secondly, it is easy to make sure that 
our uncertainty principle does not hold for an arbitrary function.

\begin{remark}
    Inequalities \eqref{f:rho,M_cor}, \eqref{f:rho,M} have no place for any function $f$, namely, the bounds
\begin{equation}\label{f:uncertainty_f}
    \| \FF{f} \|_\infty \cdot \| f\|_{\infty} \cdot |\supp f| \ge \| f\|_1^2 
    \,,
\end{equation}
\begin{equation}\label{f:uncertainty_f_max}
     \max_{\xi\neq 0} |\FF{f}(\xi)| \cdot \max_{x\neq 0} |f(x)| \cdot |\supp f| \ge \| f\|_1^2 
    \,,
\end{equation} 
    fail. 
    Indeed,  let $p$ be a prime number, $\Gr = \Z/p\Z$,   and $f (x) = A(x) - |A|/p$ be the balanced function of a random set $A$ of size $|A| > p^{2/3+\eps}$, where $\eps>0$ (one can even take the  set of quadratic residues to get a constructive example). 
    Then the right--hand side of \eqref{f:uncertainty_f}, \eqref{f:uncertainty_f_max} is $|A|^2$ but the left--hand side of \eqref{f:uncertainty_f}, \eqref{f:uncertainty_f_max} is $O(\sqrt{|A|} p^{1+o(1)}) = o(|A|^2)$. 
\label{r:non_general}
\end{remark}

\begin{remark}
\label{r:H+L_energy}
    It is easy to see that 
    the argument of the proof of Theorem \ref{t:rho,M} gives us 
\begin{equation}\label{f:connection_Ek_E(A,D)}
    \E_k (A) \E^k (A,D) \ge \frac{|A|^{4k+1}}{K} \,,
\end{equation}
    and the last bound is tight up to constants. 
    Indeed, let $\Gr = \F_2^n$ and let $A=H + \L$, $|A|=|H||\L|$,  where $H\le \Gr$ and $|\L|=K$  be a basis (for $\Gr = \Z$ a similar construction can be found in \cite[Claim 1]{LS_popularity}). Then for $x\in H$ one has $|A_x|=|A|$ and for $x\in D\setminus H$ the following holds $|A_x| = 2|A|/K$. 
    Thus $\E(A,D) = O(|A|^3)$, $\E_k (A) = \frac{|A|^{k+1}}{K} (1+o_K(1))$ and \eqref{f:connection_Ek_E(A,D)} is sharp (up to constants). 
    
    The  example above shows (also, see \cite[Claim 1]{LS_popularity}) that set of large $A_x$ has  measure zero. Similarly, $\FF{A} (\xi) = |H| H^\perp (\xi) \FF{\L} (\xi)$ and thus the set of large Fourier coefficients of $A$ has  measure zero as well. 
    Therefore, inequality \eqref{f:rho,M} is rather delicate and has place on a set of measure zero. 
\end{remark}

\bp 

Using the arguments similar to \cite{LS_popularity} (also, see \cite{SS_higher}) we obtain an analogue of Theorem \ref{t:rho,M} for the quantity $\rho_l (A)$. 
Here and below, for simplicity, we consider the symmetric case $A=B$.

\begin{theorem}
    Let $\Gr$ be a finite abelian group, $A\subseteq \Gr$, $|A|=\d N$, $l>1$ be a positive integer, 
    and let $|A-A| = K|A|$. 
    Suppose that $|A| \ge K^{8(l-1)}$, and $|A| > 4l^4$.
    Then 
\begin{equation}\label{f:rho,M_l}
   M^{2(l-1)} (A) \rho_l (A) \ge \frac{|A|^{2l-1}}{K^{l-1}} 
   \cdot 
    \left( 1-  \frac{2l^2 }{\sqrt{|A|}} \right)  \left( 1-  \frac{12 (l-1) \log^2 K}{\log |A|}  - K^2 \d \right)^{l-1}
   \,.
\end{equation}
\label{t:rho,M_l}
\end{theorem}
\begin{proof}
    We use the notation and the argument of the proof of Theorem \ref{t:rho,M}. 
    Let $\rho_l = \rho_l (A)$.
    Applying 
    the H\"older inequality, we obtain an analogue of calculations in \eqref{tmp:05.04_2}, \eqref{tmp:05.04_3}  (see  \cite[proof of Theorem 4]{LS_popularity}) 
\[
    |A|^{ln} = \left( \sum_{x_1,\dots,x_{n-1}} R^{(n)}_A (x_1,\dots, x_{n-1}) \right)^l 
    \le |A^{n-1} - \D_{n-1} (A)|^{l-1} \E_{l,n} (A) 
\]
\[
    \le 
    \eps^{l-1}_{n-1} |D|^{(l-1)(n-1)} \E_{n,l} (A)
    \le \eps^{l-1}_{n-1} |D|^{(l-1)(n-1)} \left( \rho^{n-1}_l |A|^l + l^2 |A|^{n+l-2} \right)
\]
    and hence 
\[
    \rho_{l} \ge \eps^{-\frac{l-1}{n-1}}_{n-1} \cdot \frac{|A|}{K^{l-1}} \left( 1-  \frac{l^2 \eps^{l-1}_{n-1} K^{(l-1)(n-1)}}{|A|} \right)^{1/(n-1)} 
    \ge 
    \eps^{-\frac{l-1}{n-1}}_{n-1} \cdot \frac{|A|}{K^{l-1}} \left( 1-  \frac{2l^2 \eps^{l-1}_{n-1} K^{(l-1)(n-1)}}{(n-1)|A|} \right) \,,
\]
    provided  $(l-1) n \le 2^{-1} \log_K |A|$. 
    After that we repeat the proof of  Theorem \ref{t:rho,M}, namely, we put $m=n = [(2(l-1))^{-1} \log_K |A|] \ge 2$ 
    and obtain 
\[
    M^{2(l-1)} \rho_l \ge \frac{|A|^{2l-1}}{K^{l-1}} \cdot \left( 1-  \frac{2l^2 \eps^{l-1}_{n-1} K^{(l-1)(n-1)}}{(n-1)|A|} \right) 
    \left( \eps^{-\frac{1}{n(n-1)}}_{n-1}  K^{-2/n} -  \eps^{-\frac{1}{n-1}}_{n-1} \kappa \right)^{l-1} 
\]
\[
    \ge 
    \frac{|A|^{2l-1}}{K^{l-1}} \cdot 
    \left( 1-  \frac{2l^2 }{\sqrt{|A|}} \right)  \left( 1-  \frac{12 (l-1) \log^2 K}{\log |A|}  - K \kappa \right)^{l-1} \,, 
\]
    where we have used inequality \eqref{f:M}. 
    This completes the proof. 
$\hfill\Box$
\end{proof}

\section{Further generalizations}

    Now we obtain an analogue of Theorem \ref{t:rho,M} for the quantity $\rho^{(k)} (A)$, $k>2$. 
    First of all, we give a trivial lower bound for $\rho^{(k)} (A)$.
    Let $k=2s \ge 4$ for simplicity. 
    Using
    the Cauchy--Schwartz inequality, 
    one has 
\begin{equation}\label{f:rho^k_CS}
    \frac{|A|^{k+2}}{|sA \pm A|} \le \T_{s+1} (A) = \sum_{x} A^{(k)} (x) A^{(2)} (x) 
    \le 
    \T_{s} (A) |A| + \rho^{(k)} (A) |A|^2 \,.
\end{equation}
    Hence, if we ignore the first term in \eqref{f:rho^k_CS} (it is easy to see that it is negligible for small $|sA \pm A|$), 
    then we derive 
\begin{equation}\label{tmp:15.04_1}
    \rho^{(k)} (A) \ge \frac{|A|^{k}}{|sA \pm A|} \cdot (1-o(1)) 
\end{equation}
    and this 
    estimate 
    can be considered as a trivial lower bound for the quantity $\rho^{(k)} (A)$.

\begin{theorem}
    Let $\Gr$ be a finite abelian group, $A\subseteq \Gr$ and let $k=2s\ge 2$ be a positive integer. 
    Also, let 
    $|A|=\d N$,  
     $|sA \pm A| = K|A|$ and $|sA - sA| = K_* |A|$.
    Suppose that $$|A|^{nk-n+1}  \ge  K^8_* \T^{n}_s (A) \,,$$ 
    and 
    $|A| > 4$. 
    Then 
\begin{equation}\label{f:rho,M_k}
   M^2 (A) \rho^{(k)} (A) \ge \frac{|A|^{k+1}}{K} 
    \cdot \left( 1-  \frac{14 \log^2 K_*}{\log |A|}  - \d K K_* \right) 
   \,.
\end{equation}
\label{t:rho,M_k}
\end{theorem}
\begin{proof}
     We use 
     the argument of the proof of Theorem \ref{t:rho,M}.  
     Let 
     $\rho = \rho^{(k)} (A)$, $D=A-A$, $|sD| = K_* |A|$ and put 
\begin{equation*}
    |s(A^{n-1} - \D_{n-1} (A))| = |(sA)^{n-1} - \D_{n-1} (sA)| := \eps_{n-1} |sD|^{n-1} \,,
\end{equation*}
    where 
\[
    \frac{1}{K^n_*} \le \frac{|sA|^n}{|sD|^n} \le \eps_n \le 1 \,.
\]
    Then by the Cauchy--Schwarz inequality and Lemma \ref{l:R_T_k} one has 
\begin{equation}\label{tmp:15.04_2}
    |A|^{nk} \le  |(sA)^{n-1} - \D_{n-1} (sA)| \T_{s} (R^{(n)}_A)
    =
    \eps_{n-1} |sD|^{n-1} \cdot \sum_{x} (A^{(k)} (x))^n
\end{equation}
\begin{equation}\label{tmp:15.04_3}
    \le 
    \eps_{n-1} |sD|^{n-1} (\rho^{n-1} |A|^k + \T^n_s (A)) \,, 
\end{equation}
    and hence
\[
    \rho \ge \eps^{-\frac{1}{n-1}}_{n-1} \cdot \frac{|A|^{k-1}}{K_*} \left( 1-  \frac{\eps^{}_{n-1} K^{n-1}_* \T^n_s (A)}{|A|^{nk-n+1}} \right)^{1/(n-1)} 
    \ge 
    \eps^{-\frac{1}{n-1}}_{n-1} \cdot \frac{|A|^{k-1}}{K_*} \left( 1-  \frac{2\eps^{}_{n-1} K^{n-1}_* \T^n_s (A)}{(n-1)|A|^{nk-n+1}} \right) \,,
\]
    provided $n\le 2^{-1} \log_{K_*} (|A|^{nk-n+1} \T^{-n}_s(A))$. 
    Now let $Q = sA \pm A$, $|Q| = K|A|$. 
    Then the inclusion of Katz--Koester \eqref{f:Katz-Koester} gives us 
    $sA \pm A_x \subseteq Q_{\pm x}$.
    Hence as in the proof of  Theorem \ref{t:rho,M}, we get 
\[
    \sigma:= \sum_x |A_x| |Q_x| \ge \eps^{1/m}_{m-1} |sD|^{1-1/m} \sum_x |A_x|^{1+1/m} \ge \eps^{1/m}_{m-1} K_* |A|^3 K_*^{-2/m} \,.
\]   
On the other hand, there is the upper bound \eqref{tmp:07.04_1} for $\sigma$.
Putting  $$m=n = [2^{-1} \log_{K_*} (|A|^{nk-n+1} \T^{-n}_s(A))] \ge 2$$ 
we can repeat the 
argument 
of  the proof of Theorem \ref{t:rho,M}.   
$\hfill\Box$
\end{proof}

\bp

\begin{remark}
    If $\Gr = \Z/p\Z$, where $p$ is a prime number and $|sA-sA| \le \min \{2(p+1)/3, |sA|^{1+o(1)}\}$, then 
    it is possible to obtain 
    another 
    lower bound for the quantity $\rho^{(k)} (A)$, namely, 
\begin{equation}\label{tmp:15.04_2*}
    \rho^{(k)} (A) \ge \frac{2|A|^{k}}{|sA - sA|} \cdot (1-o(1)) \,.
\end{equation}
    Indeed, combine 
    the 
    estimate 
    $|(sA)^{n-1} - \D_{n-1} (sA)| \le (1/2+o(1))^{n-1} |sD|^{n-1}$ 
    from \cite[Theorem 1, Proposition 1]{LS_popularity} 
    and calculations in \eqref{tmp:15.04_2}, \eqref{tmp:15.04_3}
    above. 
    Estimate \eqref{tmp:15.04_2*} shows that there are some irregularities in the distribution of the quantity $\rho^{(k)} (A)$ and 
    may be of interest in itself.
\end{remark}

\bp 

In the 
last 
result,  we make no   assumptions about the size of $A\pm A$, but only use the additive energy of $A$. The resulting estimate  \eqref{f:rho,M_energy} is not as accurate as in Theorems \ref{t:rho,M}, \ref{t:rho,M_l}, \ref{t:rho,M_k} but nevertheless it is optimal 
up to logarithms. 

\begin{theorem}
    Let $\Gr$ be a finite abelian group, $A\subseteq \Gr$, 
    $|A|=\d N$, and $\E(A) = |A|^3/K$.  
    Suppose that 
    $|A| \ge 8 K^3$, 
    and 
\begin{equation}\label{cond:rho,M_energy}
    \d^3 L^{28} K^{25}
    \ll 1 \,.
\end{equation}
    Then 
\begin{equation}\label{f:rho,M_energy}
    \rho^7 (A) M^4 (A) \log^7 |A| \gg \frac{|A|^{11}}{K^7}
   \,.
\end{equation}
\label{t:rho,M_energy}
\end{theorem}
\begin{proof}
    Let $L=\log |A|$, $\rho = \rho (A)$ and $M = M(A)$. 
    Also, let $\D>0$ be a real number and $P = \{ s ~:~ \D < |A_s| \le 2\D \}$ be a set such that $\D^{10/7} |P| \gg \E_{10/7} (A)/L$. 
    It is easy to see that the number $\D$ and the set $P$ exist by the pigeonhole principle. 
    Clearly, we have 
\[
    \E(A) \rho^{-4/7} L^{-1} \ll \E_{10/7} (A) L^{-1} \ll \D^{10/7} |P| \ll |A|^2 \D^{3/7} \,,
\]
    and thus 
\begin{equation}\label{tmp:16.04_1}
    \D \gg \frac{|A|^{7/3}}{(KL)^{7/3} \rho^{4/3}} \,.
\end{equation}
    Computing the eigenvalues of the operator $M(x,y) := P(x-y)A(x)A(y)$  (see, e.g., \cite[Section 8]{SS_higher} or \cite[Theorem 6.3]{sh_new_ineq}), we have 
\[
    (\Delta |P|)^8 \ll |A|^8 \E_4 (A) \E(P)
\]
    and hence using the simple bound $\rho \ge |A|/(2K)$ as well as  the assumption $|A| \ge 8K^3$, we obtain in view of \eqref{F_Par},  \eqref{f:F_svertka} 
\[
    \Delta^{10} |P|^8 \ll |A|^8 \E_4 (A) \sum_{x} P^{(2)} (x) A^{(4)} (x) \le  
    |A|^8  \left( |A|^4 + \rho^3 |A|^2 \right) \cdot  
    N^{-1} \sum_{\xi} |\FF{P} (\xi)|^2 |\FF{A} (\xi)|^4
\]
\begin{equation}\label{tmp:14.04_1}
    \ll 
    \rho^3 |A|^{10} \cdot \left( N^{-1} |P|^2 |A|^4 + M^4 |P| \right) 
\end{equation}
    If the second term in \eqref{tmp:14.04_1} dominates, then thanks to our choice of the set $P$,
    we get 
\[
    \E^7 (A) \rho^{-4} L^{-7} \ll \E^7_{10/7} (A) L^{-7} \ll \Delta^{10} |P|^7 \ll |A|^{10} \rho^3 M^4
\]
    and \eqref{f:rho,M_energy} follows. 
    If the first term in \eqref{tmp:14.04_1} is the largest one, then we derive 
\[
   \Delta^{10} |P|^6 \ll \d \E_4 (A) |A|^{11} \ll \d |A|^{11} \rho^{18/7} \E_{10/7}(A) \,,
\]
    and therefore in view of \eqref{tmp:16.04_1} 
\[
    \left( \frac{|A|^{7/3}}{(KL)^{7/3} \rho^{4/3}} \right)^{10/7}
    \E^5_{} (A) \rho^{-20/7}
    \ll 
    \D^{10/7}
    \E^5_{10/7} (A) \ll \d |A|^{11} \rho^{18/7} L^6 \,.
\]
    It follows that 
\[
    \d^3 L^{28} K^{25} \gg 1 
\]
    and this is a contradiction with \eqref{cond:rho,M_energy}. 
    This completes the proof. 
$\hfill\Box$
\end{proof}

\bibliographystyle{abbrv}

\bibliography{bibliography}{}

\begin{thebibliography}{10}

\bibitem{alagic2008uncertainty}
G.~Alagic.
\newblock {\em {Uncertainty principles for compact groups}}.
\newblock University of Connecticut, 2008.

\bibitem{biro2021uncertainty}
A.~Bir{\'o} and V.~F. Lev.
\newblock {Uncertainty in finite planes}.
\newblock {\em Journal of Functional Analysis}, 281(3):109026, 2021.

\bibitem{folland1997uncertainty}
G.~B. Folland and A.~Sitaram.
\newblock {The uncertainty principle: a mathematical survey}.
\newblock {\em Journal of Fourier analysis and applications}, 3:207--238, 1997.

\bibitem{Katz-Koester}
N.~H. Katz and P.~Koester.
\newblock On additive doubling and energy.
\newblock {\em SIAM Journal on Discrete Mathematics}, 24(4):1684--1693, 2010.

\bibitem{LS_2.6}
V.~F. Lev and I.~D. Shkredov.
\newblock Small doubling in prime-order groups: from 2.4 to 2.6.
\newblock {\em J. Number Theory}, 217:278--291, 2020.

\bibitem{LS_popularity}
V.~F. Lev and I.~D. Shkredov.
\newblock {The popularity gap}.
\newblock {\em Journal of Algebraic Combinatorics}, 58(4):1155--1172, 2023.

\bibitem{meshulam2006uncertainty}
R.~Meshulam.
\newblock {An uncertainty inequality for finite abelian groups}.
\newblock {\em European Journal of Combinatorics}, 27(1):63--67, 2006.

\bibitem{SS_higher}
T.~Schoen and I.~D. Shkredov.
\newblock Higher moments of convolutions.
\newblock {\em J. Number Theory}, 133(5):1693--1737, 2013.

\bibitem{sh_new_ineq}
I.~D. Shkredov.
\newblock Some new inequalities in additive combinatorics.
\newblock {\em Mosc. J. Comb. Number Theory}, 3(3-4):189--239, 2013.

\bibitem{Tao_uncertainty}
T.~Tao.
\newblock {An uncertainty principle for cyclic groups of prime order}.
\newblock {\em Mathematical Research Letters}, 12(1):121--127, 2005.

\bibitem{terras1999fourier}
A.~Terras.
\newblock {\em {Fourier analysis on finite groups and applications}}, volume~43.
\newblock Cambridge University Press, 1999.

\bibitem{wigderson2021uncertainty}
A.~Wigderson and Y.~Wigderson.
\newblock {The uncertainty principle: variations on a theme}.
\newblock {\em Bulletin of the American Mathematical Society}, 58(2):225--261, 2021.

\end{thebibliography}

\end{document}